\documentclass[11pt,draft,amsmath,amscd,amsbsy,amssymb]{amsart}
 \relax

\chardef\bslash=`\\ 

\makeatletter \def\verbatim{\interlinepenalty\@M \@verbatim
\leftskip\@totalleftmargin\advance\leftskip2pc
\frenchspacing\@vobeyspaces \@xverbatim} \makeatother \hfuzz1pc

\makeatletter \def\dgt@k{\dg@DX=-3 \dg@DY=2 \dg@SIZE=3}
\makeatother

\makeatletter \def\dgt@kk{\dg@DX=3 \dg@DY=-1 \dg@SIZE=3}

\theoremstyle{plain} \newtheorem{thm}{Theorem}[section]

\newtheorem{lemma}[thm]{Lemma}
\newtheorem{prop}[thm]{Proposition}

\theoremstyle{definition} \newtheorem{rem}[thm]{Remark}
\newtheorem{defin}[thm]{Definition} 

 \newtheorem{exam}[thm]{Example}

\newcommand{\R}{\mathbb R_{\mathrm{max}}}

\newcommand{\supp}{\operatorname{\mathrm{supp}}}
\newcommand{\comp}{\operatorname{\mathbf{Comp}}}

\newcommand{\diam}{\operatorname{\mathrm{diam}}}

\newcommand{\id}{\mathrm{id}}

\usepackage[all]{xy}

\begin{document}

\title[Idempotent  measures on compact metric spaces]
{Spaces of idempotent  measures of compact metric spaces}

\author{Lidia Bazylevych}
\address{Department of Mechanics and Mathematics, Lviv National University,  Universtytetska Str. 1,
79000 Lviv, Ukraine}
\email{izar@litech.lviv.ua}

\author{Du\v{s}an~Repov\v{s}}
\address{Institute of Mathematics, Physics and Mechanics, and Faculty of Education,
University of Ljubljana, P.O.B. 2964, Ljubljana, 1001, Slovenia}
\email{dusan.repovs@guest.arnes.si}

\author{Michael Zarichnyi}
\address{Department of Mechanics and Mathematics, Lviv National University,  Universtytetska Str. 1,
79000 Lviv, Ukraine}
\email{mzar@litech.lviv.ua}

\thanks{}


\date{March 29, 2008}



\begin{abstract} We investigate certain
geometric properties of the
spaces of idempotent  measures. In particular, we
prove that
the space of idempotent  measures on an infinite compact metric
space is homeomorphic to the Hilbert cube.
\end{abstract}

\maketitle
\section{Introduction}

The functor $P$ of probability measures which acts on the category
$\comp$ of compact metrizable spaces 
has been
investigated by many
authors 
(see e.g., the survey 
\cite{FF}). Geometric
properties of spaces of the form $P(X)$ were established, e.g., in
\cite{F}. In particular, it was proved in \cite{F} that the map
$P(f)\colon P(X)\to P(Y)$ is a trivial $Q$-bundle (i.e., a trivial
bundle whose fiber is the Hilbert cube $Q$) for an open map of
finite-dimensional compact metric spaces with infinite fibers.

 The space of idempotent probability measures was
systematically studied in \cite{Z} (see also \cite{Z2}), where it
was
proved, in particular, that the space $I(X)$ of idempotent
probability measures on a topological space $X$ is compact
Hausdorff if such is also $X$.
The aim of this paper, which can be considered as a continuation
of \cite{Z}, is to establish certain
geometric properties of the
functor $I$. 

In particular, we shall
prove
that $I(X)$ is homeomorphic to
the Hilbert cube for every infinite compact metric space $X$. The
construction of idempotent measures is functorial in the category
of compact Hausdorff spaces and we consider also the geometry of
the maps $I(f)$, for some maps $f$. In particular, we show that,
like in the case of probability measures, there exists an open map
$f\colon X\to Y$ of compact metric spaces such that $f$ has
infinite fibers and the map $I(f)$ is not a trivial $Q$-bundle.

The paper is organized as follows. In Section \ref{s:sim} we
provide the necessary information concerning the spaces of
idempotent measures. Section \ref{s:metr} is devoted to
(pseudo)metrization of the spaces $I(X)$, for a metric space $X$.
The main results on the topology of the spaces $I(X)$, for compact
metric spaces $X$, are given in Section \ref{s:mcomp}. We also
consider the geometry of the maps $I(f)$, for some maps $f$ of
compact metric spaces and this allows us to describe the topology
of the spaces $I(X)$ for some nonmetrizable compact Hausdorff
spaces $X$ (Section \ref{s:maps}).

\section{Preliminaries}\label{s:1}

The space $Q=\prod_{i=1}^\infty[0,1]_i$ is called the {\em Hilbert
cube}. Recall that an {\em absolute retract} (AR) is a metrizable
space which is a retract of every space in which it lies as a
closed subset. The following characterization theorem was proved in
\cite{T}.

\begin{thm}[Toru\'nczyk's characterization theorem] A compact
metric space $X$ is homeomorphic to the Hilbert cube if the
following two conditions are satisfied:
\begin{enumerate}
\item $X$ is an absolute retract;
\item $X$ satisfies the disjoint approximation property (DAP),
i.e. every two maps of a metric space into $X$ can be approximated
by maps with disjoint images.
\end{enumerate}
\end{thm}

The following notion was introduced in \cite{Ho}. A {\em
$c$-structure} on a topological space $X$ is an assignment to
every nonempty finite subset $A$ of $X$ a contractible subspace
$F(A)$ of $X$ such that $F(A)\subset F(A')$ whenever $A\subset
A'$. A pair $(X,F)$, where $F$ is a $c$-structure on $X$ is called
a {\em $c$-space}. A subset $E$ of $X$ is called an {\em $F$-set}
if $F(A)\subset E$ for any finite $A\subset E$. A metric space
$(X,d)$ is said to be a {\em metric $l.c.$-space} if all the open
balls are $F$-sets and all open $r$-neighborhoods of $F$-sets are
also $F$-sets.
In fact, it was proved in \cite{Ho1} that every compact metric
$l.c.$-space is an AR.

A map $f\colon X\to Y$ is {\em trivial $Q$-bundle} if $f$ is
homeomorphic to the projection map $p_1\colon Y\times Q\to Y$.
The following definition is due to Shchepin \cite{S}.
\begin{defin} A map $f\colon X\to Y$ is said to be {\em soft} provided
that for every commutative diagram
\begin{equation}\label{myeq}
\xymatrix{A\ar@{^{(}->}[d]\ar[r]^\varphi&X\ar[d]^f\\ Z\ar[r]_\psi&
Y}\end{equation} such that $Z$ is a paracompact space  and $A$ is
a closed subset of $Z$ there exists a map $\Phi\colon Z\to X$ such
that $f\Phi=\psi$ and $\Phi|A=\varphi$.
\end{defin}

A map $f\colon X\to Y$ of compact metric spaces is said to satisfy
the {\em fibrewise disjoint approximation property} if, for every
$\varepsilon>0$,  there exist maps $g_1,g_2\colon X\to Y$ such
that
\begin{enumerate}
\item $fg_1=fg_2=f$;
\item $d(1_X,g_i)<\varepsilon$, $i=1,2$;
\item $g_1(X)\cap g_2(X)=\emptyset$.
\end{enumerate}

The following result was proved in \cite{TW}.
\begin{thm}[Toru\'nczyk-West characterization theorem for $Q$-manifold
bundles]\label{t:bundle} A map $f\colon X\to Y$ of compact metric
ANR-spaces is a trivial $Q$-bundle if $f$ is soft and $f$
satisfies the fibrewise disjoint approximation property.
\end{thm}

The following is a
generalization of the Michael Selection Theorem --
see \cite{Ho1} for the proof. Recall that a multivalued map
$F\colon X\to Y$ of topological spaces is called {\em lower
semicontinuous} if, for any open subset $U$ of $Y$, the  set
$\{x\in X\mid F(x)\cap U\neq\emptyset\}$ is open in $X$. A {\em
selection} of a multivalued map $F\colon X\to Y$ is a
(single-valued) map $f\colon X\to Y$ such that $f(x)\in F(x)$, for
every $x\in X$.
\begin{thm}\label{t:ho} Let $(X,d,F)$ be a complete metric $l.c.$-space. Then
any lower semicontinuous multivalued map $T\colon Y\to X$ of a
paracompact space $Y$ whose values are nonempty closed $F$-sets
has a continuous selection.
\end{thm}

\section{Spaces of idempotent measures}\label{s:sim}

In the sequel, all maps will be
assumed to be continuous.
Let $X$ be a compact Hausdorff space.  
We shall
denote the
Banach space of continuous functions on $X$ endowed with the
$\sup$-norm by $C(X)$. 
For any $c\in\mathbb R$ we shall denote  the
constant function on $X$ taking the value $c$ 
by $c_X$. We shall denote
the weight of a topological space $X$ by $w(X)$.

Let $\R=\mathbb R\cup\{-\infty\}$ be
endowed with the metric
$\varrho$ defined by $\varrho(x,y)=|e^x-e^y|$. Let also
$\R^n=(\R)^n$.
Following the notation of idempotent mathematics (see
e.g.,
\cite{MS}) we shall
denote by $\odot\colon \mathbb R\times C(X)\to C(X)$
the map acting by $(\lambda,\varphi)\mapsto \lambda_X+\varphi$,
and by $\oplus\colon C(X)\times C(X)\to C(X)$ the map acting by
$(\varphi,\psi)\mapsto \max\{\varphi,\psi\}$.
For each $c\in\mathbb R$ by $c_X$ we shall
denote the constant function
from $C(X)$ defined by the formula $c_X(x)=c$ for each $x\in X$.
\begin{defin}\label{d:1} A functional $\mu\colon C(X)\to\mathbb R$ is called an {\em
idempotent probability measure} (a {\em Maslov measure}) if
\begin{enumerate}
\item $\mu(c_X)=c$;
\item $\mu(c\odot\varphi)=c\odot\mu(\varphi)$;
\item $\mu(\varphi\oplus\psi)=\mu(\varphi)\oplus\mu(\psi),$
\end{enumerate}
for every $\varphi,\psi\in C(X)$.
\end{defin}

The number $\mu(\varphi)$ is the {\em Maslov integral} of
$\varphi\in C(X)$ with respect to $\mu$.
Let $I(X)$ denote the set of all idempotent probability measures
on $X$. We endow $I(X)$ with the weak* topology. A basis of this
topology is formed by the sets
$$\langle\mu; \varphi_1,\dots,\varphi_n; \varepsilon\rangle=\{\nu\in I(X)\mid
|\mu(\varphi_i)-\nu(\varphi_i)|<\varepsilon,\ i=1,\dots,n\},$$
where $\mu\in I(X)$, $\varphi_i\in C(X)$, $i=1,\dots,n$,  and
$\varepsilon>0$.

The following is an example of an idempotent probability measure.
Let $x_1,\dots,x_n\in X$ and $\lambda_1,\dots,\lambda_n\in\R$ be
numbers such that $\max\{\lambda_1,\dots,\lambda_n\}=0$. Define
$\mu\colon C(X)\to\mathbb R$ as follows:
$\mu(\varphi)=\max\{\varphi(x_i)+\lambda_i\mid i=1,\dots,n\}$.  As
usual, for every $x\in X$, we denote by $\delta_x$ (or
$\delta(x)$) the functional on $C(X)$ defined as follows:
$\delta_x(\varphi)=\varphi(x)$, $\varphi\in C(X)$ (the Dirac
probability measure concentrated at $x$). Then one can write
$\mu=\oplus_{i=1}^n\lambda_i\odot\delta_{x_i}$.

Given a map $f\colon X\to Y$ of compact Hausdorff spaces, the map
$I(f)\colon I(X)\to I(Y)$ is defined by the formula
$I(f)(\mu)(\varphi)=\mu(\varphi f)$, for every $\varphi\in C(Y)$.
That $I(f)$ is continuous and that $I$ is a covariant functor
acting in the category $\comp$ is proved in \cite{Z}. Note that,
if $\mu=\oplus_{i=1}^n\lambda_i\odot\delta_{x_i}\in I(X)$, then
$I(f)(\mu)=\oplus_{i=1}^n\lambda_i\odot\delta_{f(x_i)}\in I(Y)$.

\subsection{Milyutin maps}

The following result was
proved in \cite{Z}.
\begin{thm}\label{t:mil} Let $X$ be a compact metrizable space. Then there
exists a zero-dimensional compact metrizable space $X$ and a
continuous map $f\colon X\to Y$ for which there exists a
continuous map $s\colon Y\to I(X)$ such that
$\mathrm{supp}(y)\subset f^{-1}(y)$, for every $y\in Y$.
\end{thm}

A map $f$ satisfying the conditions stated above is called a {\em
Milyutin map} of idempotent measures.

\subsection{Map $\zeta_X$}

Given $M\in I^2(X)$, define the map $\zeta_X(M)\colon
C(X)\to\mathbb R$ as follows:
$\zeta_X(M)(\varphi)=M(\bar\varphi)$. Given $\varphi\in C(X)$,
define $\bar\varphi\colon I(X)\to\mathbb R$ as follows:
$\bar\varphi(\mu)=\mu(\varphi)$, $\mu\in I(X)$. It is proved in
\cite{Z} that the map $\zeta_X$ is continuous.

\section{Metrization}\label{s:metr}

Let $(X,d)$ be a compact metric space.
By $\mathrm{n-LIP}=\mathrm{n-LIP}(X,d)$ we denote the set of
Lipschitz functions with the Lipschitz constant $\le n$ from
$C(X)$.
Fix $n\in\mathbb N$. For every $\mu,\nu$, let $$\hat
d_n(\mu,\nu)=\sup\{|\mu(\varphi)-\nu(\varphi)|\mid \varphi\in
\mathrm{n-LIP}\}.$$
\begin{thm} The function $\hat d_n$ is a continuous pseudometric on $I(X)$.
\end{thm}
\begin{proof} We first remark that $\hat d_n$ is well-defined.
Indeed, $\sup\varphi-\inf\varphi\le n\diam X$, for every
$\varphi\in \mathrm{n-LIP}$, whence
$|\mu(\varphi)-\nu(\varphi)|\le2n\diam X$.

Obviously, $\hat d_n(\mu,\mu)=0$ and $\hat d_n(\mu,\nu)=\hat
d_n(\nu,\mu)$, for every $\mu,\nu\in I(X)$.

We are going to prove that $\hat d$ satisfies the triangle
inequality. Since, for every $\varphi\in \mathrm{n-LIP}\}$ and
$\mu,\nu,\tau\in I(X)$, $$\hat d(\mu,\nu)\ge
|\mu(\varphi)-\nu(\varphi)|,\ \hat d(\nu,\tau)\ge
|\nu(\varphi)-\tau(\varphi)|,$$ we have $$\hat d_n(\mu,\nu)+\hat
d(\nu,\tau)\ge|\mu(\varphi)-\nu(\varphi)|+
|\nu(\varphi)-\tau(\varphi)|\ge |\mu(\varphi)-\tau(\varphi)|,$$
whence, passing to $\sup$ in the right-hand side, we obtain $\hat
d_n(\mu,\nu)+\hat d(\nu,\tau)\ge\hat d(\mu,\tau)$.

Now, we prove that $\hat d$ is continuous. Suppose to
the contrary.
Then one can find a sequence $(\mu_i)_{i=1}^\infty$  in $I(X)$
such that $\lim_{i\to\infty}\mu_i=\mu\in I(X)$ and $\hat
d(\mu_i,\mu)\ge c'$, for some $c'>0$. Then there exist
$\varphi_i\in\mathrm{n-LIP}$, $i\in \mathbb N$,  such that
$|\mu_i(\varphi_i)-\mu(\varphi_i)|\ge c$, for some $c>0$. Since
the functionals in $I(X)$ are weakly additive, without loss of
generality, one may assume that $\varphi_i(x_0)=0$, for some base
point $x_0\in X$, $i\in \mathbb N$. By the Arzela-Ascoli theorem,
there exists a limit point $\varphi\in\mathrm{n-LIP} $ of the
sequence  $(\varphi_i)_{i=1}^\infty$. We have
$|\mu_i(\varphi)-\mu(\varphi)|\ge c$, which contradicts to the
fact that $(\mu_i)_{i=1}^\infty$ converges to $\mu$.
\end{proof}
\begin{rem} Simple examples demonstrate that $\hat d$ cannot be a
metric whenever $X$  consists of more than one point.
\end{rem}

\begin{prop} The family of pseudometrics $\hat d_n$, $n\in\mathbb
N$, separates the points in $I(X)$.
\end{prop}

\begin{proof} Let $\mu,\nu\in I(X)$, $\mu\neq\nu$. There exists
$\varphi\in C(X)$ such that $|\mu(\varphi)-\nu(\varphi)|>c$, for
some $c>0$. There exists $\psi\in \mathrm{n-LIP}$, for some
$n\in\mathbb N$, such that $\|\varphi-\psi\|\le(c/3)$. Then,
clearly, $|\mu(\psi)-\nu(\psi)|\ge(c/3)$ and therefore $\hat
d_n(\mu,\nu)\ge(c/3)$.
\end{proof}

We let $\tilde d_n=(1/n)\hat d_n$.

\begin{prop} The map $\delta=\delta_X$, $x\mapsto\delta_x\colon
(X,d)\to(I(X),\tilde d_n)$, is an isometric embedding for every
$n\in\mathbb N$.
\end{prop}

\begin{proof} Let $x,y\in X$ and $\varphi\in \mathrm{n-LIP}$. Then
$|\delta_x(\varphi)-\delta_y(\varphi)|\le n d(x,y)$, therefore
$\hat d_n(\delta_x,\delta_y)\le n d(x,y)$. Thus $\tilde
d_n(\delta_x,\delta_y)\le  d(x,y)$.

On the other hand, define $\varphi_x\in\mathrm{n-LIP}$ by the
formula $\varphi_x(z)=nd(x,z)$, $z\in X$. Then
$|\delta_x(\varphi_x)-\delta_y(\varphi_x)|=nd(x,y)$ and we are
done.
\end{proof}

\begin{prop} Let $f\colon (X,d)\to (Y,\varrho)$ be a nonexpanding
map of compact metric spaces. Then the map $I(f)\colon (I(X),\hat
d_n)\to (I(Y),\hat\varrho_n)$ is also nonexpanding, for every
$n\in\mathbb N$.
\end{prop}
\begin{proof} Given $\varphi\in\mathrm{n-LIP}(Y)$, note that $\varphi
f\in\mathrm{n-LIP}(X)$ and, for any $\mu,\nu\in I(X)$, we have
$$|I(f)(\mu)(\varphi)-I(f)(\nu)(\varphi)|=|\mu(\varphi f)-\nu(\varphi f)|\le \hat d_n(\mu,\nu).$$
Passing to the limit in the left-hand side of the above formula,
we are done.
\end{proof}

Note that the above construction of $\hat d$ can be applied not
only to metrics but also to continuous pseudometrics.  Proceeding
in this way we obtain the iterations $(I(X),\tilde d_n)$,
$(I^2(X),\tilde{\tilde{d}}_{nm}=(\tilde{d}_n)\tilde{}_m)$,\dots

\begin{prop} For a metric space $(X,d)$, the map $\zeta_X\colon (I^2(X),\tilde{\tilde{d}}_{nn})
\to (I(X),\tilde d_n)$ is nonexpanding.
\end{prop}
\begin{proof} We first prove that, for any $\varphi\in\mathrm{n-LIP}(X,d)$, we have
$\bar\varphi\in \mathrm{n-LIP}(I(X),\hat d)$. Indeed, given
$\mu,\nu\in I(X)$, we see that $$n\tilde d(\mu,\nu)= \hat
d(\mu,\nu)\ge
|\mu(\varphi)-\nu(\varphi)|=|\bar\varphi(\mu)-|\bar\varphi(\nu)|$$
and we are done.

Suppose now that $M,N\in I^2(X)$, $\mu=\zeta_X(M)$,
$\nu=\zeta_X(N)$. Given $\varphi\in\mathrm{n-LIP}(X,d)$, we obtain
$$|\mu(\varphi)-\nu(\varphi)|=|M(\bar\varphi)-N(\bar\varphi)|\le \tilde{\tilde{d}}_{nn}(M,N).$$
Passing to the limit in the left-hand side, we are done.

\end{proof}

\begin{rem} Using the results on existence of the pseudometrics
$\tilde d_n$, one can define the spaces of idempotent probability
measures with compact support for metric and, more generally,
uniform spaces. Indeed, let $(X,d)$ be a metric space. We define
the set $I(X)$ to be the direct limit of the direct system
$\{I(A),I(\iota_{AB});\exp X\}$ (here, for $A,B\in \exp X$ with
$A\subset B$, we denote by $\iota_{AB}\colon A\to B$ the inclusion
map). For every $A\in\exp X$, we identify $I(A)$ with the
corresponding subset of $I(X)$ along the map $I(\iota_A)$, where
$\iota_A\colon A\to X$ is the limit inclusion map. For any $\mu\in
I(X)$, there exists a unique minimal $A\in\exp X$ such that
$\mu\in I(A)$. Then we say that $A$ is the {\em support} of $\mu$
and write $\mathrm{supp}(\mu)=A$.

Now, define a family of pseudometrics $\hat d_n$, $n\in\mathbb N$,
on $I(X)$ as follows. Given $\mu,\nu\in I(X)$, we let $$\hat
d_n(\mu,\nu)= \hat
d_n|((\mathrm{supp}(\mu)\cup\mathrm{supp}(\nu))\times(\mathrm{supp}(\mu)\cup\mathrm{supp}(\nu)))(\mu,\nu).$$
One can prove that, for any uniform space $(X,\mathcal U)$, if the
uniformity $\mathcal U$ is generated by a family $\{d^\alpha\mid
\alpha\in A\}$ of pseudometrics, then the family $\{\tilde
d^\alpha_n\mid \alpha\in A,\ n\in\mathbb N\}$ of pseudometrics on
$I(X)$ generates a uniformity on $I(X)$.
\end{rem}

Let $(X,d)$ be a compact metric space. We define a function
$\tilde d\colon I(X)\times I(X)\to \mathbb R$ as follows:
$$\tilde d(\mu,\nu)=\sum_{i=1}^{\infty}\frac{\tilde d_n(\mu,\nu)}{2^i}.$$
It follows from what was proved above that $\tilde d$ is a metric
on the space $I(X)$.

\section{Space of idempotent measures for metric
compacta}\label{s:mcomp}
 It is proved in \cite{Z} that the set
$I(X)$ is homeomorphic to the $(n-1)$-dimensional simplex for any
finite $X$ with $|X|=n$.

A set $A\subset I(X)$ is called {\em max-plus convex} if, for
every $\mu,\nu\in A$ and every $\alpha,\beta\in\R$ with
$\alpha\oplus\beta=0$, we have
$\alpha\odot\mu\oplus\beta\odot\nu\in A$.

\begin{lemma} Let $\mu_0\in I(X)$. The map $h\colon I(X)\times [-\infty,0]\to
I(X)$, $h(\mu,\lambda)=\mu\oplus(\lambda\odot\mu_0)$, is
continuous.
\end{lemma}
\begin{proof} Let $(\mu,\lambda)\in I(X)\times [-\infty,0]$,
$\nu=h(\mu,\lambda)$, and $\langle\nu;\varphi;\varepsilon\rangle$
be a subbase neighborhood of $\nu$.

Case 1). $h(\mu,\lambda)=\mu(\varphi)$. Then
$\mu(\varphi)\ge\lambda+\mu_0(\varphi)$ and it is evident that,
for any $\mu'\in\langle\mu;\varphi;\varepsilon\rangle$ and
$\lambda'\in[-\infty,\lambda+\varepsilon)\cap \R$, we have
$h(\mu',\lambda')\in\langle\nu;\varphi;\varepsilon\rangle$.

Case 2). $h(\mu,\lambda)=\lambda+\mu_0(\varphi)$. Then necessarily
$\lambda>-\infty$. For every
$\mu'\in\langle\mu;\varphi;\varepsilon\rangle$ and
$\lambda'\in(\lambda+\varepsilon,\lambda+\varepsilon)\cap \R$, we
have $h(\mu',\lambda')\in\langle\nu;\varphi;\varepsilon\rangle$.
\end{proof}
\begin{lemma} Let $X$ be a compact metrizable space. Every max-plus convex subset in $I(X)$ is
contractible.
\end{lemma}
\begin{proof} Let $A\subset I(X)$ be a nonempty max-plus convex subset in
$I(X)$. Note that $\max A\in A$, because of max-plus convexity of
$A$. Define the map $H\colon A\times \R\to A$ as follows:
$H(\mu,\lambda)=\mu\oplus(\lambda\odot \max A)$. Then
$H(\mu,-\infty)=\mu$ and $H(\mu,0)=\max A$. This shows that the
set $A$ is contractible.
\end{proof}
\begin{thm} The space $I(X)$ is homeomorphic to the Hilbert cube
for any infinite compact metrizable space $X$.
\end{thm}
\begin{proof}

We first show that $I(X)$ is an AR-space. Fix a metric $d$ on $X$
that generates its topology. Define a $c$-structure on the space
$I(X)$ as follows.  To every nonempty finite subset
$A=\{\mu_1,\dots,\mu_n\}$ of $I(X)$ assign a  subspace
$$F(A)=\{\oplus_{i=1}^n\alpha_i\odot\mu_i\mid
\alpha_1,\dots,\alpha_n\in \R,\ \oplus_{i=1}^n\alpha_i=0\}.$$ It
is easy to verify that $F(A)$ is max-plus convex and therefore
contractible.

We are going to show that this $c$-structure is an
$l.c.$-structure. We shall prove
that every ball with
respect to the metric $\hat d$ in $I(X)$ is convex.

Let $\mu,\nu,\tau\in I(X)$, $\lambda\in\R$, and $\varphi\in
\mathrm{n-LIP}$. Then
$$|\mu(\varphi)-((\lambda\odot\nu)\oplus\tau)(\varphi)|\le\max\{|\mu(\varphi)-\nu(\varphi)|,
|\mu(\varphi)-\tau(\varphi)|\},$$ whence $$\hat d_n(\mu,
(\lambda\odot\nu)\oplus\tau)\le \max\{\hat d_n(\mu,\nu), \hat
d_n(\mu,\tau)\}$$ and therefore, for any $\varepsilon>0$, if $\hat
d(\mu,\nu)<\varepsilon$ and $\hat d(\mu,\tau)<\varepsilon$, then
$\hat d(\mu, (\lambda\odot\nu))<\varepsilon$.

We now
show that the space $I(X)$ satisfies DAP. Let $f\colon Y\to
X$ be a Milyutin map of a zero-dimensional compact metrizable
space $Y$. There exists a continuous map $s\colon X\to I(Y)$ such
that $I(f)s(x)=\delta_x$, for every $x\in X$. Let $r\colon Y\to
Y'$ be a retraction of $Y$ onto its finite subset $Y'$. Since $Y$
is zero-dimensional, we may choose $r$ as close to $\id_Y$ as we
wish. Define $g_1\colon I(X)\to I(X)$ as follows:
$$g_1(\mu)=I(fr)\zeta_YI(s)(\mu)=\zeta_YI^2(fr)I(s)(\mu).$$

Then $g_1(I(X))\subset I_\omega(X)$.

Recall that $j_X(X)\in I(X)$ acts as follows:
$j_X(X)(\varphi)=\max\varphi$, $\varphi\in C(X)$. Define
$g_2\colon I(X)\to I(X)$ by the formula
$g_2(\mu)=\mu\oplus\lambda\odot j_X(X)$, where  $\lambda\in
(-\infty,0]$ is fixed. If $\lambda$ is small enough, the map $g_2$
is close to $\id_{I(X)}$.

Note that $\supp(g_2(\mu))=X$, for every $\mu\in I(X)$, and
therefore $g_1(I(X))\cap g_2(I(X))=\emptyset$. By Toru\'nczyk's
theorem, $I(X)$ is homeomorphic to $Q$.
\end{proof}

\begin{rem} Note that, for max-plus convex sets, a version of the
Michael Selection Theorem is proved in \cite{Z1}.
\end{rem}

\section{Maps of spaces of idempotent measures}\label{s:maps}

\begin{thm}\label{t:soft} Let $p_1\colon X\times Y\to X$ denote the projection
onto the first factor, where $X$, $Y$ are compact metric spaces.
Then the map $I(p_1)\colon I(X\times Y)\to I(X)$ is soft.
\end{thm}
\begin{proof} It is proved in \cite{Z} that the map $I(f)$ is
open. Given a commutative diagram
$$\xymatrix{A\ar@{^{(}->}[d]\ar[r]^\varphi&I(X\times Y)\ar[d]^{I(p_1)}\\
Z\ar[r]_\psi& I(X)}$$
\end{proof}

\begin{exam} The following example demonstrates that, like in the
case of probability measures, there exists an open map $f\colon
X\to Y$ of metrizable compacta with infinite fibers such that the
map $I(f)\colon I(X)\to I(Y)$ is not a trivial $Q$-bundle.

We exploit the construction from \cite{D}. Let $S^n$ denote the
$n$-dimensional sphere and ${\mathbb R}P^n$ the $n$-dimensional
real projective space. Let $\eta_n\colon S^n\to{\mathbb R}P^n$
denote the canonical map. The required map is
$$f=\prod_{i=1}^\infty \eta_{2^i-1}\colon X=\prod_{i=1}^\infty
S^{2^i-1}\to Y=\prod_{i=1}^\infty {\mathbb R}P^{2^i-1}.$$ It is
proved in \cite{D} that the map
$$f_k=\prod_{i=1}^k \eta_{2^i-1}\colon \prod_{i=1}^k S^{2^i-1}\to
\prod_{i=1}^k {\mathbb R}P^{2^i-1}$$ has the following property:
the map $P_0(f_k)\colon P_0(\prod_{i=1}^k S^{2^i-1})\to
\prod_{i=1}^k {\mathbb R}P^{2^i-1}$, where
$$P_0\left(\prod_{i=1}^k
S^{2^i-1}\right)=(P(f_k))^{-1}\left(\left\{\delta_x\mid
x\in\prod_{i=1}^k {\mathbb R}P^{2^i-1} \right\}\right)$$ and
$P_0(f_k)$ sends every $\mu\in P_0\left(\prod_{i=1}^k
S^{2^i-1}\right)$ to the unique $x\in\prod_{i=1}^k {\mathbb
R}P^{2^i-1} $ for which $\supp(\mu)\in f_k^{-1}(x)$, has no two
disjoint selections.

Proceeding similarly as in \cite{D} we reduce the problem of
existence of two disjoint sections of the map $I(f)$ to that of
existence of two disjoint selections of the map $I_0(f_k)$, for
some $k$, where the map $I_0(f_k)$ is defined similarly as
$P_0(f_k)$ with $P$ replaced by $I$.

We only
have 
to show that the maps $I_0(f_k)$ and $P_0(f_k)$ are
homeomorphic. Let $\mu=\oplus_{i=1}^n\lambda_i\odot
\delta_{x_i}\in P_0(f)$. Then $$(e^{\lambda_1},\dots,e^{\lambda_n}
)\in \Gamma^n=\{(z_1,\dots,z_n)\in[0,1]^n\mid z_1\oplus\dots\oplus
z_n=1\}.$$ Let $$\Sigma^{n-1}=\{(z_1,\dots,z_n)\in\Gamma^n\mid
z_i=0\text{ for some }i\},$$ then there exists a  point
$(\Lambda_1,\dots,\Lambda_n)\in \Sigma^{n-1}$ such that
$$(e^{\lambda_1},\dots,e^{\lambda_n}
)=s(\Lambda_1,\dots,\Lambda_n)+(1-s)(1,\dots,1),$$ for some $s\in
[0,1]$ (note that such a point $(\Lambda_1,\dots,\Lambda_n)$ is
uniquely determined whenever $(e^{\lambda_1},\dots,e^{\lambda_n}
)\neq(1,\dots,1)$). We then let
$$(\lambda_1',\dots,\lambda_n')=s(\Lambda_1',\dots,\Lambda_n')+(1-s)(1/n,\dots,1/n),$$
where $(\Lambda_1',\dots,\Lambda_n')$ is the point of intersection
of the linear segment 
$$[(\Lambda_1,\dots,\Lambda_n),(1,\dots,1)]$$
with the boundary $$\partial\Delta^{n-1}=\{(x_1,\dots,x_n)\mid
\sum_{x=1}^nx_i=1 \}.$$

One can easily verify that $\mu\mapsto
\mu'=\sum_{i=1}^n\lambda_i'\delta_{x_i}\in P_0(f)$ is a desired
homeomorphism.

\end{exam}

\section{Nonmetrizable case}
The notion of normal functor in the category $\comp$  of compact
Hausdorff spaces was introduced in \cite{S}.

\begin{thm} Let $\tau>\omega_1$. Then the set $I([0,1]^\tau)$ is
not an AR.
\end{thm}
\begin{proof} In \cite{S}, Shchepin proved that, for any normal
functor $F$  which is not a power functor, any cardinal number
$\tau>\omega_1$, and any compact metric space $K$ with $|K|\ge2$,
the functor-power $F(K^\tau)$ is not an AR. In \cite{Z}, it is
proved that $I$ is a normal functor, whence the result follows.
\end{proof}

\begin{thm} Let $X$ be an openly generated character-homogeneous
compact Hausdorff space of weight $\omega_1$. Then the space
$I(X)$ is homeomorphic to $I^{\omega_1}$.
\end{thm}
\begin{proof}
We can represent $X$ as $\varprojlim\mathcal S$, where $\mathcal
S=\{X_\alpha,p_{\alpha\beta};\omega_1\}$ is an inverse system such
that $p_{\alpha\beta}\colon X_\alpha\to X_\beta$ are open maps for
all $\alpha,\beta<\omega_1$, $\alpha\ge\beta$, and $X_\alpha$,
$\alpha<\omega_1$, are infinite compact metric spaces. Since $X$
is character-homogeneous, we may additionally assume that the maps
$p_{\alpha\beta}$ do not contain singleton fibers.

Then we have $I(X)=\varprojlim I(\mathcal
S)=\{I(X_\alpha),I(p_{\alpha\beta});\omega_1\}$ (see \cite{Z}). By
Theorem \ref{t:soft}, the maps $I(p_{\alpha\beta})$ are soft.
Applying arguments  from \cite{S1} one can find a cofinal subset
$A$ of $\omega_1$ such that, for every $\alpha,\beta\in A$,
$\alpha>\beta$, the map $I(p_{\alpha\beta})$ satisfies the FDAP.
Therefore, by the Toru\'nczyk-West theorem, the map
$I(p_{\alpha\beta})$ is homeomorphic to the projection
$\pi_1\colon Q\times Q\to Q$ onto the first factor. In turn,
$I(X)$ is homeomorphic to $Q^A\simeq Q^{\omega_1}\simeq
I^{\omega_1}$.

\end{proof}

\section{Epilogue}

One can also consider the spaces $I(K^\tau)$, for arbitrary $\tau$
and nondegenerate compact metrizable space $K$. The interesting
results on autohomeomorphisms of the spaces $P(K^\tau)$, for
$\tau>\omega_1$, were
obtained by Smurov \cite{Sm}. One can
conjecture that these results have their counterparts also in the
case of spaces of idempotent measures.

In connection with the mentioned in the introduction result by
Fedorchuk, the following question arises: 
{\it Is the map
$I(f)\colon I(X)\to I(Y)$ is a trivial $Q$-bundle, for an open map
of finite-dimensional compact metric spaces with infinite fibers?}

As it was remarked above, one can also consider the spaces $I(X)$
for noncompact metric space $X$. It looks plausible that the
results on topology of spaces of probability measures proved in
\cite{BC} should
have their counterparts also for the idempotent
measures.

\section*{Acknowledgements}
This research was supported by SRA grants P1-0292-0101-04 and  BI-UA/07-08-001.

\end{document}